%% file: z1.tex
\documentclass[11pt]{article}

\usepackage{amsmath,amssymb,amsthm,color,graphicx}
\usepackage{subfigure}

\graphicspath{{.},{figures/}}

\newtheorem{theorem}{Theorem}
\newtheorem{lemma}{Lemma}

\newtheorem{proposition}{Proposition}

\newtheorem{question}{Question}

\theoremstyle{definition}

\newcommand{\comment}[1]{}
\newcommand{\ignore}[1]{}

\def\clap#1{\hbox to 0pt{\hss#1\hss}}

\makeatletter
  \def\moverlay{\mathpalette\mov@rlay}
  \def\mov@rlay#1#2{\leavevmode\vtop{%
    \baselineskip\z@skip \lineskiplimit-\maxdimen
    \ialign{\hfil$#1##$\hfil\cr#2\crcr}}}
\makeatother

\newcommand\A{\mathcal{A}}

\def\A{{\cal A}}

\newcommand{\remove}[1]{}

\begin{document}

\title{On the zone complexity of a vertex}

\author{
Shira Zerbib\thanks{Department of Mathematics,
      Technion---Israel Institute of Technology,
      Haifa 32000, Israel.
      \texttt{zgshira@tx.technion.ac.il}.
} }

\date{\today}

\maketitle

\begin{abstract}
Let $L$ be a set of $n$ lines in the real projective plane in
general position. We show that there exists a vertex $v\in \A(L)$
such that $v$ is positioned in a face of size at most 5 in the
arrangement obtained by removing the two lines passing through $v$.
\end{abstract}

\section{Introduction}
Let $L$ be a set of $n$ lines in the real projective plane in
general position (i.e., no three of which pass through the same
point), and let $\A(L)$ be the planar arrangement of vertices, edges
and faces determined by $L$. The \emph{zone} $Z(\ell)$ of a line
$\ell \in L$ is defined as the collection of faces supported by
$\ell$. The \emph{complexity} $C(\ell)$ of the zone of $\ell$ is the
sum of the sizes of all the faces in $\A(L\backslash \{\ell\})$
which contain a segment of $\ell$. The zone theorem asserts that
$C(\ell) \leq 5.5(n-1) + O(1)$, and that this bound is tight
\cite{1}. This was an improvement of the upper bound of $6(n-1)$
given in \cite{2} and \cite{3}.

A natural and interesting question that arises is whether the bound
given by the zone theorem can be improved on average. Namely, we
ask:

\begin{question}\label{aveZone}
Let $L$ be a set of $n$ lines in the real projective plane in
general position. Is it true that $\frac{1}{n}\sum_{\ell \in L}
C(\ell) \leq c(n-1)$, where $c$ is a positive constant strictly
smaller than 5.5?
\end{question}

Looking for supporting evidence to a positive answer to Question
\ref{aveZone}, we were naturally led to define the notion of the
zone complexity of a vertex. For a vertex $v \in \A(L)$, let
$\ell^v_1,\ell^v_2$ be the two lines passing through $v$. We define
the \emph{zone} $Z(v)$ of $v$ as the collection of four faces
containing $v$. The \emph{complexity} $C(v)$ of the zone of $v$ is
defined as the size of the face in the arrangement $\A(L \backslash
\{\ell^v_1,\ell^v_2\})$ which contains the position of $v$ in the
plane. Let $C(L)$ denote the minimal vertex zone complexity in
$\A(L)$.

An easy consequence of the zone theorem is that every line $\ell \in
L$ passes through a vertex $v$ with $C(v) \leq 7$ (see Section
\ref{zone} for a proof of this claim). This shows in particular that
$C(L) \leq 7$ for any arrangement $L$. Now, if the answer to
Question \ref{aveZone} is positive, it will similarly indicate that
there must be a vertex $v \in \A(L)$ such that $C(v) < 7$, i.e.,
$C(L) < 7$. Our main result shows that this is indeed the case,
namely we prove the following:

\begin{theorem}\label{main}
For every set $L$ of $n$ lines in the real projective plane in
general position, $C(L) \leq 5$.
\end{theorem}

Recall the standard representation of the real projective plane as a
quotient space of the sphere, where any two antipodal points are
identified. In this model, lines are represented by great circles on
the sphere, and the crossing point of two lines is the pair of
antipodal points at which the two respective great circles meet. The
notions of the zone of a vertex and its complexity are translated on
the sphere in an obvious way. Thus for a set $L$ of $n$ great
circles on the sphere in general position we can define the minimal
vertex complexity $C(L)$ as above. Translating Theorem \ref{main},
we can state our main result in the following equivalent way:

\begin{theorem}\label{main1}
For every set $L$ of $n$ great circles on the sphere in general
position, $C(L) \leq 5$.
\end{theorem}

Theorem \ref{main1} is proved in Section \ref{proof}. Moreover, in
Section \ref{example} we give an example of a set of lines $L$ in
the real projective plane for which every vertex $v\in \A(L)$ has
$C(v) \geq 5$, which shows that the bound given in Theorem
\ref{main} is tight. Note that since the zone theorem gives a tight
upper bound on the line zone complexity, the upper bound given in
Theorem \ref{main} cannot be achieved using the zone theorem by
considering each line individually. Therefore Theorem \ref{main} may
suggest that the answer to Question \ref{aveZone} is positive.

\section{An upper bound on $C(L)$ using the zone theorem}\label{zone}

The notion of the zone complexity of a vertex is closely related to
that of the zone complexity of a line. In fact, the zone complexity
$C(\ell)$ of a line $\ell$ can be expressed in terms of the zone
complexities of the vertices that lie on $\ell$. In the following
proposition we calculate this relation and prove, using the zone
theorem, that $C(L) \leq 7$.

\begin{proposition}
For every set $L$ of $n$ lines in the real projective plane in
general position, $C(L) \leq 7$.
\end{proposition}
\begin{proof}
 We write $v \in \ell$ if $v$ lies
on $\ell$ (that is, if $\ell$ passes through $v$), and let $|f|$ be
the size of a face $f$. The relation between the sizes of the faces
in $Z(\ell)$ and the sizes of the faces in the zones of the vertices
that lie on $\ell$ is given in the following equation:
\begin{equation}\label{1}
    \sum_{v\in \ell} \sum_{f \in Z(v)} |f| = 2\sum_{f\in Z(\ell)}|f|.
\end{equation}

Observe that for every vertex $v \in \A(L)$,
\begin{equation}\label{1.1}
    \sum_{f \in Z(v)} |f| = C(v) + 12.
\end{equation}
Indeed, since $v$ contributes 4 to $\sum_{f \in Z(v)} |f|$ and each
of the 4 vertices adjacent to $v$ contributes 2 to $\sum_{f \in
Z(v)} |f|$, $v$ and its neighbors contribute 12.

Now, since the number of vertices that lie on $\ell$ is $n-1$,
(\ref{1.1}) implies
\begin{equation}\label{2}
    \sum_{v\in \ell}\sum_{f\in Z(v)}|f| = \sum_{v \in \ell}C(v) + 12(n-1).
\end{equation}
On the other hand, it is easy to see that
\begin{equation}\label{3}
    \sum_{f\in Z(\ell)} |f| = C(\ell) + 4(n-1).
\end{equation}
Therefore, combining (\ref{1}), (\ref{2}) and (\ref{3}) we get the
following relation between the zone complexity of a line $\ell$ and
the zone complexities of the vertices that lie on it:
\begin{equation}\label{4}
    C(\ell) = \frac{1}{2}\sum_{v\in \ell} C(v) + 2(n-1).
\end{equation}

Finally, set $r(\ell) := min_{v\in \ell} C(v)$. From (\ref{4}) and
the zone theorem we get that
\begin{equation}\label{5}
  \frac{1}{2}(n-1)\cdot r(\ell) + 2(n-1) \leq  C(\ell) \leq 5.5(n-1) +
  O(1).
\end{equation}
Since the constant $O(1)$ in the zone theorem is actually negative
(it equals $-1$; see the proof of Theorem 1 in \cite{1}), it follows
from (\ref{5}) that $r(\ell) \leq 7$ for all $n$. Hence every line
in $L$ passes through a vertex $v$ with $C(v)\leq 7$, and in
particular $C(L) \leq 7$, as claimed.
\end{proof}

\section{Proof of Theorem \ref{main1}}\label{proof}

Throughout the proof we use the term \emph{4-multiset} to describe a
multiset of cardinality 4. We begin by proving an elementary
technical lemma which will play a crucial role in the sequel.

\begin{lemma}\label{negative multisets}
Let $K=\{k_1,k_2,k_3,k_4\}$ be a 4-multiset of integers with the
following properties:
\begin{enumerate}
  \item $k_i \geq 3$ for $i = 1,2,3,4$.
  \item At most two of the elements in $K$ equal 3.
  \item $\sum_{i=1}^4 k_i \geq 18.$
\end{enumerate}
Then $\sum_{i=1}^4 \frac{k_i-3}{k_i} < 1 $ if and only if $K$ is one
of the following 4-multisets:
$$\{3,3,4,8\}, \ \{3,3,4,9\}, \ \{3,3,4,10\}, \ \{3,3,4,11\}, \ \{3,3,5,7\}.$$
\end{lemma}
\begin{proof}
Let $K=\{k_1,k_2,k_3,k_4\}$ be a 4-multiset with the above
properties. Then $\sum_{i=1}^4 \frac{k_i-3}{k_i} < 1$ if and only if
$\sum_{i=1}^4 \frac{1}{k_i} > 1$. Change the order in $K$ such that
$k_1\leq k_2 \leq k_3 \leq k_4$. It follows from Property (2) that
$k_3 \geq 4$. Moreover, $k_4 \geq 6$ or $k_3 \geq 5$, since
otherwise $\sum_{i=1}^4 k_i \leq 4+4+4+5 = 17$ in contradiction to
Property (3). If $k_2 \geq 4$ we have
$$\sum_{i=1}^4 \frac{1}{k_i} \leq \frac{1}{3} + \frac{1}{4} +
\frac{1}{4} + \frac{1}{6}= 1.$$ Thus $\sum_{i=1}^4 \frac{1}{k_i} >
1$ implies $k_1=k_2=3$. It is now straightforward to verify that
there are only 5 possibilities for $\{k_3,k_4\}$: $\{4,8\}$,
$\{4,9\}$, $\{4,10\}$, $\{4,11\}$ or $\{5,7\}$, as claimed.

\comment{For such a multiset $K=\{k_1,k_2,k_3,k_4\}$ denote  
$$\Sigma K := \sum_{i=1}^4 k_i, \ \ \ \ \widehat{K} :=
\sum_{i=1}^4 \frac{k_i-3}{k_i}.$$ Let $T$ be the set of all the
4-multisets with the properties of the lemma. By Property 3
$$T = \bigcup_{r\geq 18} T_r,$$ where $T_r$ is the finite set
$$T_r = \{K : K \in T, \ \Sigma K=r\}.$$ We claim that $min\{ \widehat{K}:K \in T_r\} < min\{
\widehat{K}:K \in T_{r+1}\}$. Indeed, take $K=\{k_1,k_2,k_3,k_4\}
\in T_{r+1}$ such that $\widehat{K}$ is minimal, and change the
order in $K$ such that $k_1\leq k_2 \leq k_3 \leq k_4$. Then $\Sigma
K \geq 18$ implies $k_4 > 4$. Therefore the 4-multiset
$K':=\{k_1,k_2,k_3,k_4-1\}$ is in $T_r$. Observe that the inequality
$\widehat{K'} < \widehat{K}$ holds, therefore,
$$min\{ \widehat{K}:K \in T_r\} \leq \widehat{K'} < \widehat{K} =  min\{
\widehat{K}:K \in T_{r+1}\}.$$ Now, for a given $r$, calculating
$min\{ \widehat{K}:K \in T_r\}$ is an easy task since $T_r$ is
finite. Thus one can easily verify that $min\{\widehat{K}:K \in
T_{22}\} = 1$, which immediately implies $\widehat{K} \geq 1$ for
all the 4-multisets $K$ such that $\Sigma K \geq 22$. Therefore all
the 4-multisets $K$ such that $\widehat{K} < 1$ are in the finite
(and rather small) set
$$\bigcup_{r=18}^{21} T_r.$$ A quick check shows that such a 4-multiset
has to be one of the five given in the lemma.}

\end{proof}

Let $L$ be a set of $n$ great circles in general position on the
sphere. In order to prove that $C(L) \leq 5$ we need to show that
there exists a vertex $v \in \A(L)$ such that $C(v) \leq 5$. Assume
to the contrary that every vertex $v \in \A(L)$ has $C(v) \geq 6$.
By (\ref{1.1}), this assumption is equivalent to
\begin{equation}\label{6}
    \sum_{f\in Z(v)} |f| \geq 18
\end{equation}
for every $v$.

Denote by $V$ the number of vertices, by $E$ the number of edges,
and by $F$ the number of faces in the planar arrangement $\A(L)$. By
Euler's formula, we have $V - E + F = 2$. For every $k\geq 3$ denote
by $f_k$ the number of faces in $\A(L)$ of size $k$. We observe that
$4V = 2E = \sum kf_k$ and $F = \sum f_k$. Therefore,
\begin{equation}\label{7}
-6 = -3V + E + 2E - 3F = -3V + 2V + \sum kf_k - 3\sum f_k = -V +
\sum (k-3)f_k.
\end{equation}

We are going now to use the discharging method. The discharging
method is a technique often used to prove statements in structural
graph theory, and is commonly applied in the context of planer
graphs (for a review on the discharging method and some of its
applications see \cite{RT}, \cite{ABKPR08}). Our plan is to assign
to every face and vertex of the arrangement $\A(L)$ an initial
charge, such that the sum of all assigned charges is negative. Then
we redistribute (discharge) the charges in two steps, such that
after these two steps every face and vertex in $\A(L)$ will have a
nonnegative charge. It will thus follow that the total initial
charge is nonnegative, which is a contradiction.

\textbf{Step 1 [initial charging]:} We begin by assigning a charge
$w_1(\cdot)$ to the faces and vertices of the arrangement $\A(L)$:
The charge of a face of size $k$ is $k-3$, while the charge of each
vertex is $-1$. It follows from (\ref{7}) that the overall charge is
$-6$.

\textbf{Step 2 [discharging the faces]:} For every $k\geq 3$, every
face $f$ of size $k$ contributes a charge of $$\frac{w_1(f)}{k} =
\frac{k-3}{k}$$ to each of its $k$ vertices. After this step the
charge of each face is 0. Denote by $w_2(\cdot)$ the charge of the
vertices after Step 2.

For a vertex $v\in \A(L)$ denote by $K_v$ the 4-multiset of the
sizes of the faces in $Z(v)$.

\begin{proposition}\label{negative vertices}
Let $v$ be a vertex in $\A(L)$. Then $w_2(v)<0$ if and only if $K_v$
is one of the following: $\{3,3,4,8\}, \ \{3,3,4,9\}, \
\{3,3,4,10\}, \ \{3,3,4,11\}$ or $\{3,3,5,7\}.$
\end{proposition}
\begin{proof}
If we exclude the case $n \leq 3$ (for which Theorem \ref{main1} is
trivial), two faces of size 3 cannot share an edge. Hence there are
at most two faces of size 3 in $Z(v)$. In addition, by (\ref{6}),
$\sum_{f\in Z(v)} |f| \geq 18$. Therefore the 4-multiset $K_v$
satisfies the conditions of Lemma \ref{negative multisets}. We
deduce that $\sum_{f\in Z(v)} \frac{|f|-3}{|f|} < 1$ if and only if
$K_v$ is one of the five 4-multisets listed in the proposition. The
result follows now since
$$w_2(v) = -1 + \sum_{f\in Z(v)} \frac{|f|-3}{|f|}.$$
\end{proof}

We say that two vertices $v$ and $u$ are \emph{neighbors} if
$\{v,u\}$ is an edge in $\A(L)$.

For a vertex $u \in \A(L)$ such that $w_2(u) \geq 0$, denote by
$V^{-}_u$ the set of all vertices $v$ in $\A(L)$ with the following
two properties:
\begin{enumerate}
  \item $w_2(v) < 0$, and
  \item $v$ and $u$ are neighbors, or $v$ and $u$ are opposite vertices
        in a face of size 4.
\end{enumerate}

\textbf{Step 3 [discharging vertices with positive charge]:} A
vertex $u\in L$ such that $w_2(u) \geq 0$ and $V^{-}_u \neq \phi$
contributes a charge of $\frac{w_2(u)}{|V^{-}_u|}$ to each one of
the vertices in $V^{-}_u$. Denote by $w_3(\cdot)$ the charge of the
vertices after Step 3. The next proposition completes the proof of
Theorem \ref{main1}:

\begin{proposition}\label{all is nonnegative}
For every vertex $v \in \A(L)$, $w_3(v) \geq 0$.
\end{proposition}
\begin{proof}
Clearly, $w_2(v) \geq 0$ implies $w_3(v) \geq 0$. We show that
$w_3(v) \geq 0$ also in the case $w_2(v) < 0$. Let $v\in \A(L)$ be a
vertex such that $w_2(v) < 0$. By Proposition \ref{negative
vertices}, $K_v$ is one of the following: $\{3,3,4,8\}, \
\{3,3,4,9\}, \ \{3,3,4,10\}, \ \{3,3,4,11\}$ or $\{3,3,5,7\}$. We
split into two cases:

\textbf{Case 1.} $K_v = \{3,3,4,k\}$ for some $8\leq k \leq 11$. In
this case
$$w_2(v) \geq -1 + \frac{1}{4} + \frac{5}{8} = -\frac{1}{8}.$$

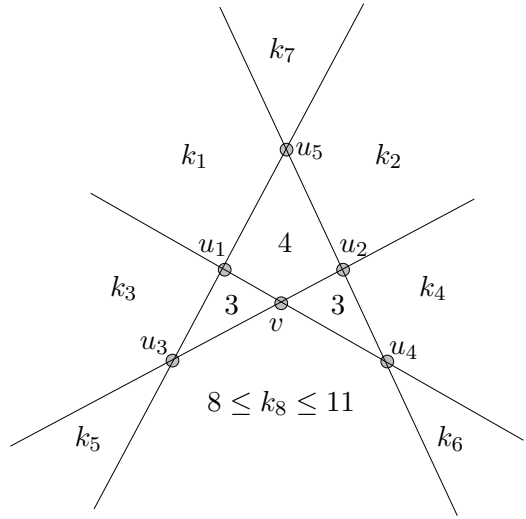
\begin{figure}[ht]
\begin{center}
\input{fig1new.pstex_t}
\caption{The local neighborhood of $v$ in Case 1.} \label{fig1}
\end{center}
\end{figure}

Figure \ref{fig1} illustrates the local neighborhood of $v$, i.e.,
$k_i$ ($i=1,\dots,8$) denote the sizes of the faces $f_i$ in this
neighborhood, and $u_i$ ($i=1,\dots,5$) denote the vertices. Note
that $k_3 \geq 4$ and $k_4 \geq 4$ because $f_3$ and $f_4$ share an
edge with a face of size 3. Moreover, $k_1 \geq 4$. Indeed, if
$k_1=3$ then $f_1$ is supported by the antipodal vertex of $u_4$,
hence $L$ contains only 4 great circles which is a contradiction to
$k_8 \geq 8$ (see Figure \ref{fig3}). A symmetric argument yields
$k_2 \geq 4$. Therefore, by Proposition \ref{negative vertices},
$w_2(u_1) \geq 0$, $w_2(u_2) \geq 0$ and $w_2(u_5) \geq 0$, as the
zone of each contains at most one face of size 3. Observe that $u_1,
u_2, u_3, u_4$ are neighbors of $v$, and that $u_5$ and $v$ are
opposite vertices in a face of size 4. Therefore, for every $1\leq i
\leq5$, if $w_2(u_i) \geq 0$ then $v\in V^{-}_{u_i}$.

\begin{figure}[ht]
\begin{center}
\input{fig3new.pstex_t}
\caption{If $k_1=3$ the antipodal vertex of $u_4$ supports
$f_1$.}\label{fig3}
\end{center}
\end{figure}
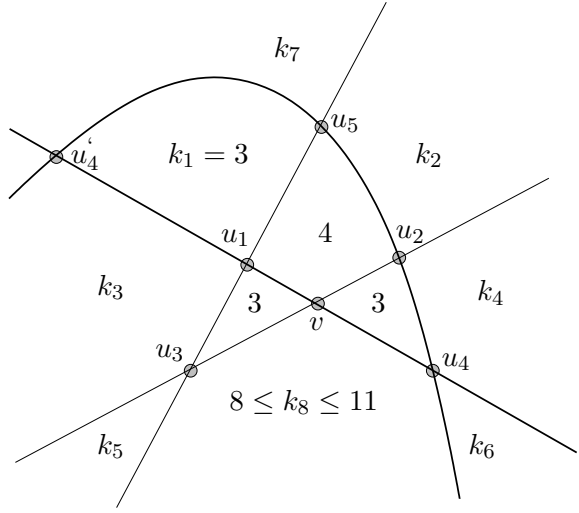

We shell now consider separately 4 possible subcases:

\textbf{Subcase 1.1.} $w_2(u_3) < 0$ and $w_2(u_4) < 0$. By
Proposition \ref{negative vertices}, $k_5=k_6=3$ and $k_3=k_4=4$.
Hence, by (\ref{6}), we have $k_1 \geq 7$ and $k_2 \geq 7$.
Therefore,
$$w_2(u_5) \geq -1 + \frac{1}{4} + \frac{4}{7} + \frac{4}{7} =
\frac{11}{28}.$$

We claim that $|V^{-}_{u_5}| \leq 3$. Indeed, since $v \in
V^{-}_{u_5}$ we have to show that there are at most two more
vertices in $V^{-}_{u_5}$. We have $w_2(u_1), w_2(u_2) \geq 0$, and
hence $u_1, u_2 \notin V^{-}_{u_5}$. Moreover, $k_1, k_2 \geq 7$, so
$k_1, k_2 \neq 4$, and thus none of the vertices in $f_1$ and $f_2$
is an opposite vertex to $u_5$ in a face of size 4. There are three
options: If $k_7=3$ then at most two vertices of $f_1$ are in
$V^{-}_{u_5}$ (the two neighbors of $u_5$), and if $k_7 = 4$ then at
most one vertex of $f_1$ is in $V^{-}_{u_5}$ (the opposite vertex to
$u_5$). In both cases the neighbors of $u_5$ have a positive weight
by Proposition \ref{negative vertices}, and hence they are not in
$V^{-}_{u_5}$. Finally, if $k_7 \geq 5$ none of the vertices of
$f_1$ is in $V^{-}_{u_5}$, and the claim is proved (see Figure
\ref{case11fig}).

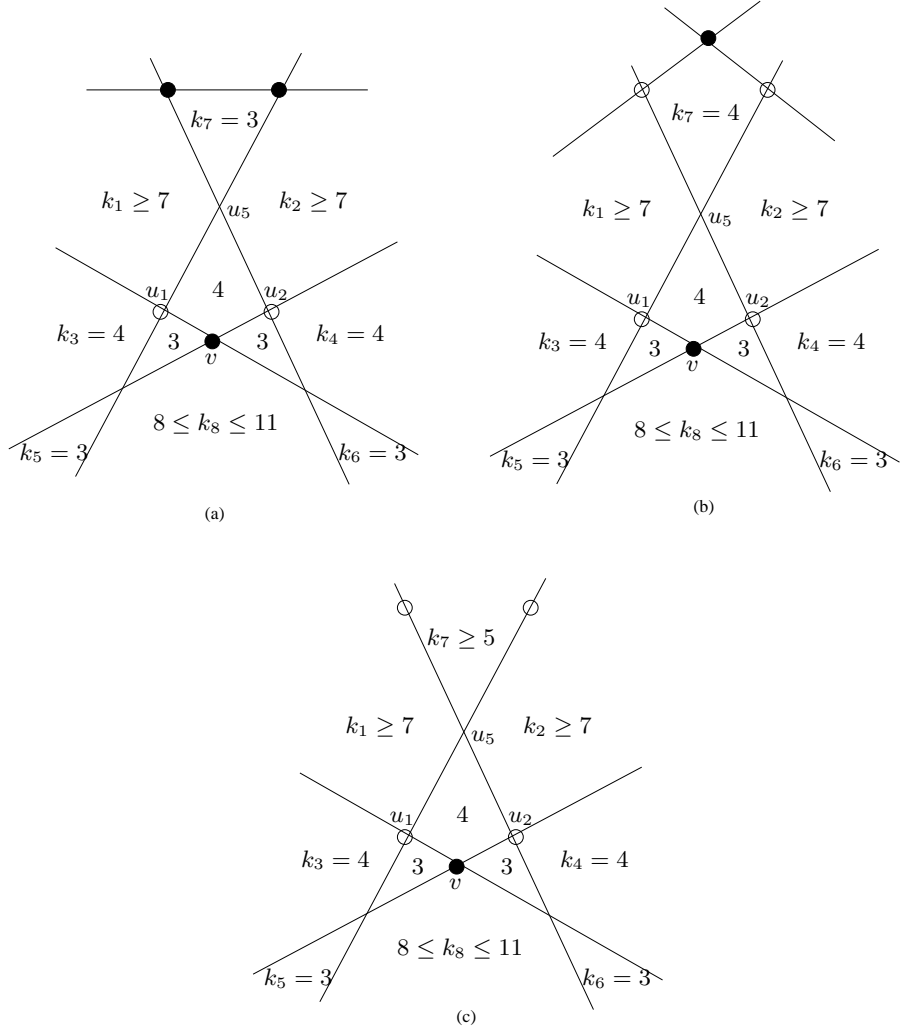
\begin{figure}[ht]
\begin{center}
\input{case11.pstex_t}
\caption{Three different possible scenarios in Subcase 1.1, in which
$|V^{-}_{u_5}| \leq 3$. Vertices with $w_2(v) \geq 0$ are marked by
empty circles, while vertices $v$ which might have $w_2(v) < 0$ are
marked by full black circles.}\label{case11fig}
\end{center}
\end{figure}

Since $v \in V^{-}_{u_5}$, we have
$$w_3(v) \geq w_2(v) + \frac{w_2(u_5)}{|V^{-}_{u_5}|} \geq -\frac{1}{8} + \frac{1}{3} \cdot
\frac{11}{28} = \frac{1}{168} > 0,$$ as claimed.

\textbf{Subcase 1.2.} $w_2(u_3) < 0$ and $w_2(u_4) \geq 0$. By
Proposition \ref{negative vertices}, $k_5=3$ and $k_3=4$, and hence
by (\ref{6}), $k_1 \geq 7$. Similar arguments to those in Subcase
1.1 yield $|V^{-}_{u_1}| \leq 3$, $|V^{-}_{u_2}| \leq 2$,
$|V^{-}_{u_4}| \leq 3$ and $|V^{-}_{u_5}| \leq 3$. Since $v \in
V^{-}_{u_1}\cap V^{-}_{u_2}\cap V^{-}_{u_4}\cap V^{-}_{u_5}$, we
have
\begin{eqnarray}\label{8}
\nonumber \lefteqn{w_3(v) = w_2(v) + \frac{w_2(u_1)}{|V^{-}_{u_1}|}
+ \frac{w_2(u_2)}{|V^{-}_{u_2}|} + \frac{w_2(u_4)}{|V^{-}_{u_4}|}+
\frac{w_2(u_5)}{|V^{-}_{u_5}|}  }
\\ \nonumber &\geq&  -\frac{1}{8} +
\frac{1}{3}\Big(-1 + \frac{1}{4} + \frac{1}{4} + \frac{4}{7}\Big) +
\frac{1}{2}\Big(-1 + \frac{1}{4} + \frac{k_2-3}{k_2} +
\frac{k_4-3}{k_4}\Big)
\\ \nonumber & & + \frac{1}{3}\Big(-1 + \frac{5}{8} +
\frac{k_4-3}{k_4} + \frac{k_6-3}{k_6}\Big) + \frac{1}{3}\Big(-1 +
\frac{1}{4} + \frac{4}{7} + \frac{k_2-3}{k_2} +
\frac{k_7-3}{k_7}\Big)
\\ &=&  -\frac{111}{168} + \frac{5}{6}\Big(\frac{k_2-3}{k_2} +
\frac{k_4-3}{k_4}\Big) + \frac{1}{3}\Big(\frac{k_6-3}{k_6} +
\frac{k_7-3}{k_7}\Big).
\end{eqnarray}
By (\ref{6}) $k_2+k_4 \geq 11$. Therefore the right hand side of
(\ref{8}) is minimal when $k_6=k_7=3$ and $\{k_2,k_4\} = \{4,7\}$.
We conclude that $$w_3(v) \geq -\frac{111}{168} +
\frac{5}{6}\Big(\frac{1}{4} + \frac{4}{7}\Big) = \frac{1}{42} > 0,$$
as claimed.

\textbf{Subcase 1.3.} $w_2(u_3) \geq 0$ and $w_2(u_4) < 0$. This
subcase is symmetric to Subcase 1.2.

\textbf{Subcase 1.4.} $w_2(u_3) \geq 0$ and $w_2(u_4) \geq 0$. Here
$|V^{-}_{u_1}| \leq 2$, $|V^{-}_{u_2}| \leq 2$, $|V^{-}_{u_3}| \leq
3$, $|V^{-}_{u_4}| \leq 3$, and $|V^{-}_{u_5}| \leq 3$. Since $v \in
\bigcap_{i=1}^5 V^{-}_{u_i}$, we have
\begin{eqnarray}\label{9}
\nonumber \lefteqn{w_3(v) = w_2(v) + \frac{w_2(u_1)}{|V^{-}_{u_1}|}
+ \frac{w_2(u_2)}{|V^{-}_{u_2}|} + \frac{w_2(u_3)}{|V^{-}_{u_3}|} +
\frac{w_2(u_4)}{|V^{-}_{u_4}|}+ \frac{w_2(u_5)}{|V^{-}_{u_5}|}  }
\\ \nonumber &\geq&  -\frac{1}{8} +
\frac{1}{2}\Big(-1 + \frac{1}{4} + \frac{k_1-3}{k_1} +
\frac{k_3-3}{k_3}\Big) + \frac{1}{2}\Big(-1 + \frac{1}{4} +
\frac{k_2-3}{k_2} + \frac{k_4-3}{k_4}\Big)
\\ \nonumber & & + \frac{1}{3}\Big(-1 + \frac{5}{8} + \frac{k_3-3}{k_3} +
\frac{k_5-3}{k_5}\Big) + \frac{1}{3}\Big(-1 + \frac{5}{8} +
\frac{k_4-3}{k_4} + \frac{k_6-3}{k_6}\Big)
\\ \nonumber & & + \frac{1}{3}\Big(-1 +
\frac{1}{4} + \frac{k_1-3}{k_1} + \frac{k_2-3}{k_2} +
\frac{k_7-3}{k_7}\Big)
\\ &=&  -\frac{11}{8} + \frac{5}{6}\Big(\frac{k_1-3}{k_1} + \frac{k_2-3}{k_2}
+ \frac{k_3-3}{k_3} + \frac{k_4-3}{k_4}\Big) +
\frac{1}{3}\Big(\frac{k_5-3}{k_5} + \frac{k_6-3}{k_6} +
\frac{k_7-3}{k_7}\Big).
\end{eqnarray}

Consider the expression
\begin{eqnarray}\label{9.9}
\frac{k_1-3}{k_1} + \frac{k_2-3}{k_2} + \frac{k_3-3}{k_3} +
\frac{k_4-3}{k_4}.
\end{eqnarray}
Recall that from (\ref{6}), $k_1+k_3 \geq 11$ and $k_2+k_4 \geq 11$,
and that $k_i \geq 4$ for $i=1,2,3,4$. Thus, if at least one of
$k_5$, $k_6$, $k_7$ is greater than or equal to 4, (\ref{9.9})
attains its minimum when $\{k_1,k_3\} = \{k_2,k_4\} = \{4,7\}$.
Therefore, in (\ref{9}),
$$w_3(v) \geq -\frac{11}{8} +
\frac{5}{6}\Big(\frac{1}{4}+\frac{4}{7}+\frac{1}{4}+\frac{4}{7}\Big)
+ \frac{1}{3}\cdot\frac{1}{4} = \frac{13}{168} > 0,$$ as claimed.
Now, assume $k_5 = k_6 = k_7 = 3$. As $w_2(u_3) \geq 0$ and
$w_2(u_4) \geq 0$, we have that $k_3\geq 5$ and $k_4\geq 5$ by
Proposition \ref{negative vertices}. Moreover, since $k_7=3$,
(\ref{6}) implies also that $k_1+k_2 \geq 11$. Thus,
$\{k_1,k_2,k_3,k_4\}$ has to be one of the following: $\{k_1 \geq
4,k_2\geq 7,k_3\geq 7,k_4\geq 5\}$, \linebreak $\{k_1\geq 5,k_2\geq
6,k_3\geq 6,k_4\geq 5\}$, $\{k_1\geq 6,k_2\geq 5,k_3\geq 5,k_4\geq
6\}$, $\{k_1\geq 6,k_2\geq 6,k_3\geq 5,k_4\geq 5\}$, or $\{k_1\geq
7,k_2 \geq 4,k_3\geq 5,k_4\geq 7\}$. A direct calculation shows that
in each one of these options (\ref{9.9}) is strictly greater than
$\frac{33}{20}$. Therefore, in (\ref{9}),
$$w_3(v) > -\frac{11}{8} +
\frac{5}{6}\cdot \frac{33}{20} = 0,$$ as claimed.

This completes the proof in Case 1.

\textbf{Case 2.} $K_v = \{3,3,5,7\}$. In this case
$$w_2(v) = -1 + \frac{2}{5} + \frac{4}{7} = -\frac{1}{35}.$$

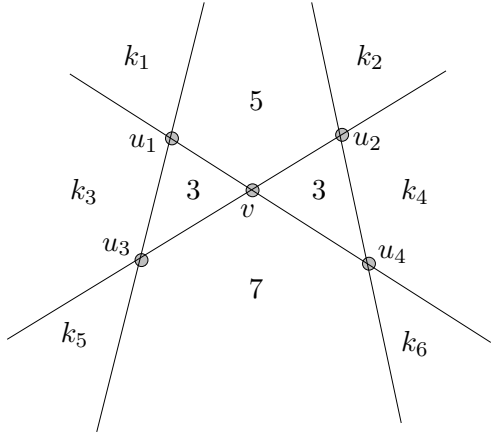
\begin{figure}[ht]
\begin{center}
\input{fig2new.pstex_t}
\caption{The local neighborhood of $v$ in Case 2.} \label{fig2}
\end{center}
\end{figure}

Figure \ref{fig2} illustrates the local neighborhood of $v$, i.e.,
$k_i$ ($i=1,\dots,6$) denote the sizes of the faces $f_i$ in this
neighborhood, and $u_i$ ($i=1,\dots,4$) denote the vertices. Observe
that $u_1, u_2, u_3, u_4$ are neighbors of $v$. Therefore, for every
$1\leq i \leq 4$, if $w_2(u_i) \geq 0$ then $v\in V^{-}_{u_i}$.

We shell show that the charge that $v$ receives in Step 3 from the
neighbors to its left ($u_1$ and/or $u_3$) is greater than
$\frac{1}{70}$ . Symmetric arguments will show that $v$ also
receives such a charge from the vertices to its right ($u_2$ and/or
$u_4$). Therefore, we shell get that
$$w_3(v) > w_2(v) + 2\cdot \frac{1}{70} = -\frac{1}{35} +
\frac{1}{35} = 0,$$ as required.

To this end we consider 3 possible subcases:

\textbf{Subcase 2.1.} $w_2(u_1) < 0$ and $w_2(u_3) \geq 0$. By
Proposition \ref{negative vertices}, $k_1 = 3$ and $k_3 = 7$. Hence,
$$w_2(u_3) \geq -1 + \frac{4}{7} + \frac{4}{7} = \frac{1}{7}.$$ Moreover,
$|V^{-}_{u_3}| \leq 4$ and $v \in V^{-}_{u_3}$. Therefore in Step 3
the charge that $u_3$ contributes to $v$ is
$$\frac{w_2(u_3)}{|V^{-}_{u_3}|} \geq \frac{1}{4} \cdot \frac{1}{7}
= \frac{1}{28} > \frac{1}{70}.$$

\textbf{Subcase 2.2.} $w_2(u_1) \geq 0$ and $w_2(u_3) < 0$. By
Proposition \ref{negative vertices}, $k_3 = 5$ and $k_5 = 3$, hence
(\ref{6}) implies $k_1 \geq 5$. Thus,
$$w_2(u_1) \geq -1 + \frac{2}{5} + \frac{2}{5} + \frac{2}{5} =
\frac{1}{5}.$$ Observe that $|V^{-}_{u_1}| \leq 4$ and that $v \in
V^{-}_{u_1}$. Therefore in Step 3 $u_1$ contributes to $v$ a charge
of
$$\frac{w_2(u_1)}{|V^{-}_{u_1}|} \geq \frac{1}{4} \cdot \frac{1}{5}
= \frac{1}{20} > \frac{1}{70},$$ as well.

\textbf{Subcase 2.3.} $w_2(u_1) \geq 0$ and $w_2(u_3) \geq 0$. We
first claim that $|V^{-}_{u_1}| \leq 3$. Indeed, since $v \in
V^{-}_{u_1}$ we have to show that there are at most two more
vertices in $V^{-}_{u_1}$. We have $w_2(u_3) \geq 0$, thus $u_3
\notin V^{-}_{u_1}$. Consider three options: If $k_1 = 3$ then by
(\ref{6}) and Proposition \ref{negative vertices}, $k_3 \geq 8$, and
therefore there are at most two more vertices in $V^{-}_{u_1}$ (the
two neighbors of $u_1$ that belong to $f_1$). If $k_1 = 4$ then by
Proposition \ref{negative vertices}, $k_3 \geq 6$, and hence there
is at most one more vertex in $V^{-}_{u_1}$ (the opposite vertex to
$u_1$ in $f_1$). Finally, consider the option $k_1 \geq 5$. Note
that $k_3 \geq 4$ since $f_3$ shares an edge with a face of size 3.
Therefore $f_1$ does not contribute any vertex to $V^{-}_{u_1}$.
Thus, there is at most one more vertex in $V^{-}_{u_1}$ (the
opposite vertex to $u_1$ in $f_3$, in case $k_4 = 4$). This proves
the claim, and similar arguments yield $|V^{-}_{u_3}| \leq 3$, as
well (see Figure \ref{case23fig}).

\begin{figure}[ht]
\begin{center}
\input{case23.pstex_t}
\caption{Four different possible scenarios in Subcase 2.3, in which
$|V^{-}_{u_1}| \leq 3$. Vertices with $w_2(v) \geq 0$ are marked by
empty circles, while vertices $v$ which might have $w_2(v) < 0$ are
marked by full black circles.} \label{case23fig}
\end{center}
\end{figure}
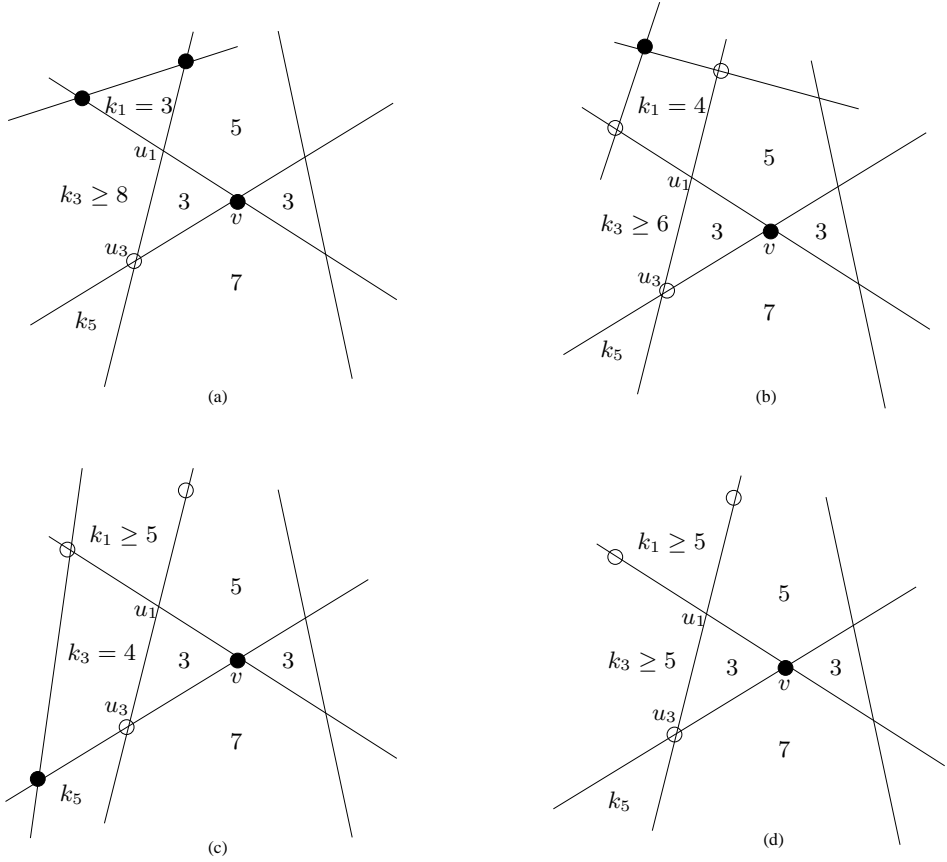

Next, we have $v \in V^{-}_{u_1}\cap V^{-}_{u_3}$, and therefore in
Step 3 $u_1$ and $u_3$ contribute to $v$ an overall charge of
\begin{eqnarray} \label{10}
\nonumber \frac{w_2(u_1)}{|V^{-}_{u_1}|} +
\frac{w_2(u_3)}{|V^{-}_{u_3}|} &\geq& \frac{1}{3}\Big(-1+
\frac{2}{5} + \frac{k_1-3}{k_1} + \frac{k_3-3}{k_3}\Big) +
\frac{1}{3}\Big(-1+ \frac{4}{7} +
\frac{k_3-3}{k_3} + \frac{k_5-3}{k_5}\Big) \\
    &=& -\frac{12}{35} + \frac{2}{3}\Big(\frac{k_3-3}{k_3}\Big) +
\frac{1}{3}\Big(\frac{k_1-3}{k_1} + \frac{k_5-3}{k_5}\Big)
\end{eqnarray}
Let us show that the right hand side of (\ref{10}) is greater than
$\frac{1}{70}$. Note that $k_3 \geq 4$, since $f_3$ shares an edge
with a face of size 3. There are 4 possible options:

\begin{enumerate}
  \item If $k_3 = 4$ then by (\ref{6}), $k_1 \geq 6$ and $k_5 \geq 4$.
Therefore the right hand side of (\ref{10}) is at least
$$-\frac{12}{35} + \frac{2}{3}\cdot \frac{1}{4} +
\frac{1}{3}\Big(\frac{1}{2} + \frac{1}{4}\Big) = \frac{31}{420} >
\frac{1}{70}.$$
  \item If $k_3 = 5$ then (\ref{6}) implies $k_1 \geq 5$ and by
Proposition \ref{negative vertices}, $k_5 \geq 4$. Thus, the right
hand side of (\ref{10}) is at least
$$-\frac{12}{35} + \frac{2}{3}\cdot \frac{2}{5} +
\frac{1}{3}\Big(\frac{2}{5} + \frac{1}{4}\Big) = \frac{59}{420} >
\frac{1}{70}.$$
  \item If $k_3 = 6$, we have from (\ref{6}) that $k_1 \geq 4$. Here the
right hand side of (\ref{10}) is at least
$$-\frac{12}{35} + \frac{2}{3}\cdot \frac{3}{6} +
\frac{1}{3}\cdot\frac{1}{4} = \frac{31}{420} > \frac{1}{70}.$$
  \item Finally, if $k_3 \geq 7$, the right hand side of (\ref{10}) is
at least
$$-\frac{12}{35} + \frac{2}{3}\cdot \frac{4}{7} = \frac{4}{105} >
\frac{1}{70}.$$
\end{enumerate}

To complete the proof in Case 2, note that the subcase where both
$w_2(u_1) < 0$ and $w_2(u_3) < 0$ is not possible, since by
Proposition \ref{negative vertices} it yields $5 = k_3 = 7$, which
is absurd.

\end{proof}

This concludes the proof of Theorem \ref{main1}.

\section{The upper bound of 5 on $C(L)$ is tight}\label{example}
In Figure \ref{fig4} we give an example of a set $L$ of 10 lines in
the real projective plane in general position, such that every
vertex $v \in \A(L)$ has $C(v) \geq 5$. The planar arrangement
$\A(L)$ has 45 vertices, 30 faces of size 3, 6 faces of size 5 and
10 faces of size 6. There are two types of vertices in $\A(L)$: 30
vertices $v$ with $K_v = \{3,3,5,6\}$ for which $C(v)=5$, and 15
vertices $v$ with $K_v = \{3,3,6,6\}$ for which $C(v)=6$. This
example shows that the upper bound of 5 on $C(L)$ given in Theorem
\ref{main1} is tight.

\begin{figure}[ht]
\begin{center}
\input{fig4.pstex_t}
\caption{A set $L$ of 10 lines such that every vertex $v\in \A(L)$
has $C(v) \geq 5$.} \label{fig4}
\end{center}
\end{figure}
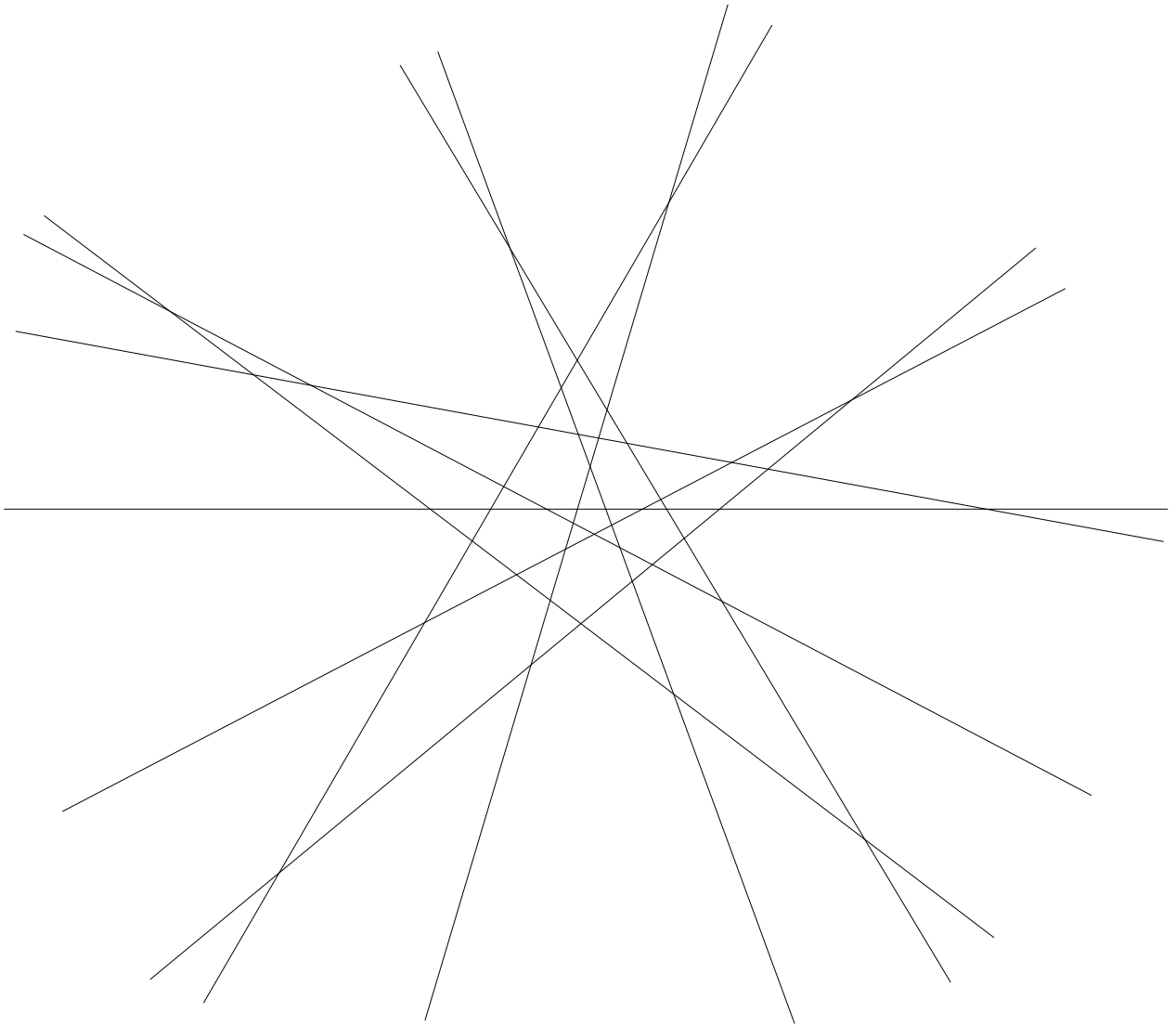

\noindent \textbf{Remark.} The example in Figure \ref{fig4} is the
only example we found of an arrangement $L$ with $C(L)=5$. It would
be interesting to know whether there are general such examples, or
that the upper bound given in Theorem \ref{main} can actually be
improved for large enough $n$.

\noindent \textbf{Acknowledgments.} The author is grateful to Rom
Pinchasi for many valuable discussions.

\end{document}

%% file: fig1new.pstex_t
\begin{picture}(0,0)%
\includegraphics{fig1new.pstex}%
\end{picture}%
\setlength{\unitlength}{2921sp}%
\begingroup\makeatletter\ifx\SetFigFont\undefined%
\gdef\SetFigFont#1#2#3#4#5{%
  \reset@font\fontsize{#1}{#2pt}%
  \fontfamily{#3}\fontseries{#4}\fontshape{#5}%
  \selectfont}%
\fi\endgroup%
\begin{picture}(4379,4373)(570,-5012)
\put(1426,-3136){\makebox(0,0)[lb]{\smash{{\SetFigFont{11}{13.2}{\rmdefault}{\mddefault}{\updefault}{\color[rgb]{0,0,0}$k_3$}%
}}}}
\put(2776,-1111){\makebox(0,0)[lb]{\smash{{\SetFigFont{11}{13.2}{\rmdefault}{\mddefault}{\updefault}{\color[rgb]{0,0,0}$k_7$}%
}}}}
\put(4051,-3136){\makebox(0,0)[lb]{\smash{{\SetFigFont{11}{13.2}{\rmdefault}{\mddefault}{\updefault}{\color[rgb]{0,0,0}$k_4$}%
}}}}
\put(2401,-3286){\makebox(0,0)[lb]{\smash{{\SetFigFont{11}{13.2}{\rmdefault}{\mddefault}{\updefault}{\color[rgb]{0,0,0}3}%
}}}}
\put(3301,-3286){\makebox(0,0)[lb]{\smash{{\SetFigFont{11}{13.2}{\rmdefault}{\mddefault}{\updefault}{\color[rgb]{0,0,0}3}%
}}}}
\put(2851,-2761){\makebox(0,0)[lb]{\smash{{\SetFigFont{11}{13.2}{\rmdefault}{\mddefault}{\updefault}{\color[rgb]{0,0,0}4}%
}}}}
\put(3676,-2011){\makebox(0,0)[lb]{\smash{{\SetFigFont{11}{13.2}{\rmdefault}{\mddefault}{\updefault}{\color[rgb]{0,0,0}$k_2$}%
}}}}
\put(1126,-4411){\makebox(0,0)[lb]{\smash{{\SetFigFont{11}{13.2}{\rmdefault}{\mddefault}{\updefault}{\color[rgb]{0,0,0}$k_5$}%
}}}}
\put(4201,-4411){\makebox(0,0)[lb]{\smash{{\SetFigFont{11}{13.2}{\rmdefault}{\mddefault}{\updefault}{\color[rgb]{0,0,0}$k_6$}%
}}}}
\put(2251,-4111){\makebox(0,0)[lb]{\smash{{\SetFigFont{11}{13.2}{\rmdefault}{\mddefault}{\updefault}{\color[rgb]{0,0,0}$8\leq k_8\leq 11$}%
}}}}
\put(2776,-3436){\makebox(0,0)[lb]{\smash{{\SetFigFont{11}{13.2}{\rmdefault}{\mddefault}{\updefault}{\color[rgb]{0,0,0}$v$}%
}}}}
\put(2176,-2761){\makebox(0,0)[lb]{\smash{{\SetFigFont{11}{13.2}{\rmdefault}{\mddefault}{\updefault}{\color[rgb]{0,0,0}$u_1$}%
}}}}
\put(3001,-1936){\makebox(0,0)[lb]{\smash{{\SetFigFont{11}{13.2}{\rmdefault}{\mddefault}{\updefault}{\color[rgb]{0,0,0}$u_5$}%
}}}}
\put(3376,-2761){\makebox(0,0)[lb]{\smash{{\SetFigFont{11}{13.2}{\rmdefault}{\mddefault}{\updefault}{\color[rgb]{0,0,0}$u_2$}%
}}}}
\put(2026,-2011){\makebox(0,0)[lb]{\smash{{\SetFigFont{11}{13.2}{\rmdefault}{\mddefault}{\updefault}{\color[rgb]{0,0,0}$k_1$}%
}}}}
\put(3796,-3631){\makebox(0,0)[lb]{\smash{{\SetFigFont{11}{13.2}{\rmdefault}{\mddefault}{\updefault}{\color[rgb]{0,0,0}$u_4$}%
}}}}
\put(1681,-3586){\makebox(0,0)[lb]{\smash{{\SetFigFont{11}{13.2}{\rmdefault}{\mddefault}{\updefault}{\color[rgb]{0,0,0}$u_3$}%
}}}}
\end{picture}%

%% file: fig3new.pstex_t
\begin{picture}(0,0)%
\includegraphics{fig3new.pstex}%
\end{picture}%
\setlength{\unitlength}{2921sp}%
\begingroup\makeatletter\ifx\SetFigFont\undefined%
\gdef\SetFigFont#1#2#3#4#5{%
  \reset@font\fontsize{#1}{#2pt}%
  \fontfamily{#3}\fontseries{#4}\fontshape{#5}%
  \selectfont}%
\fi\endgroup%
\begin{picture}(4862,4317)(357,-4956)
\put(2401,-3286){\makebox(0,0)[lb]{\smash{{\SetFigFont{11}{13.2}{\rmdefault}{\mddefault}{\updefault}{\color[rgb]{0,0,0}3}%
}}}}
\put(2251,-4111){\makebox(0,0)[lb]{\smash{{\SetFigFont{11}{13.2}{\rmdefault}{\mddefault}{\updefault}{\color[rgb]{0,0,0}$8\leq k_8\leq 11$}%
}}}}
\put(3001,-2686){\makebox(0,0)[lb]{\smash{{\SetFigFont{11}{13.2}{\rmdefault}{\mddefault}{\updefault}{\color[rgb]{0,0,0}4}%
}}}}
\put(3451,-3286){\makebox(0,0)[lb]{\smash{{\SetFigFont{11}{13.2}{\rmdefault}{\mddefault}{\updefault}{\color[rgb]{0,0,0}3}%
}}}}
\put(2176,-2686){\makebox(0,0)[lb]{\smash{{\SetFigFont{11}{13.2}{\rmdefault}{\mddefault}{\updefault}{\color[rgb]{0,0,0}$u_1$}%
}}}}
\put(2626,-1111){\makebox(0,0)[lb]{\smash{{\SetFigFont{11}{13.2}{\rmdefault}{\mddefault}{\updefault}{\color[rgb]{0,0,0}$k_7$}%
}}}}
\put(1126,-3136){\makebox(0,0)[lb]{\smash{{\SetFigFont{11}{13.2}{\rmdefault}{\mddefault}{\updefault}{\color[rgb]{0,0,0}$k_3$}%
}}}}
\put(901,-2011){\makebox(0,0)[lb]{\smash{{\SetFigFont{11}{13.2}{\rmdefault}{\mddefault}{\updefault}{\color[rgb]{0,0,0}$u_{4}^{`}$}%
}}}}
\put(4051,-3736){\makebox(0,0)[lb]{\smash{{\SetFigFont{11}{13.2}{\rmdefault}{\mddefault}{\updefault}{\color[rgb]{0,0,0}$u_4$}%
}}}}
\put(1126,-4486){\makebox(0,0)[lb]{\smash{{\SetFigFont{11}{13.2}{\rmdefault}{\mddefault}{\updefault}{\color[rgb]{0,0,0}$k_5$}%
}}}}
\put(4276,-4486){\makebox(0,0)[lb]{\smash{{\SetFigFont{11}{13.2}{\rmdefault}{\mddefault}{\updefault}{\color[rgb]{0,0,0}$k_6$}%
}}}}
\put(1726,-2011){\makebox(0,0)[lb]{\smash{{\SetFigFont{11}{13.2}{\rmdefault}{\mddefault}{\updefault}{\color[rgb]{0,0,0}$k_1=3$}%
}}}}
\put(2926,-3436){\makebox(0,0)[lb]{\smash{{\SetFigFont{11}{13.2}{\rmdefault}{\mddefault}{\updefault}{\color[rgb]{0,0,0}$v$}%
}}}}
\put(3826,-2011){\makebox(0,0)[lb]{\smash{{\SetFigFont{11}{13.2}{\rmdefault}{\mddefault}{\updefault}{\color[rgb]{0,0,0}$k_2$}%
}}}}
\put(4351,-3181){\makebox(0,0)[lb]{\smash{{\SetFigFont{11}{13.2}{\rmdefault}{\mddefault}{\updefault}{\color[rgb]{0,0,0}$k_4$}%
}}}}
\put(3661,-2656){\makebox(0,0)[lb]{\smash{{\SetFigFont{11}{13.2}{\rmdefault}{\mddefault}{\updefault}{\color[rgb]{0,0,0}$u_2$}%
}}}}
\put(3106,-1711){\makebox(0,0)[lb]{\smash{{\SetFigFont{11}{13.2}{\rmdefault}{\mddefault}{\updefault}{\color[rgb]{0,0,0}$u_5$}%
}}}}
\put(1621,-3661){\makebox(0,0)[lb]{\smash{{\SetFigFont{11}{13.2}{\rmdefault}{\mddefault}{\updefault}{\color[rgb]{0,0,0}$u_3$}%
}}}}
\end{picture}%

%% file: case11.pstex_t
\begin{picture}(0,0)%
\includegraphics{case11.pstex}%
\end{picture}%
\setlength{\unitlength}{2447sp}%
\begingroup\makeatletter\ifx\SetFigFont\undefined%
\gdef\SetFigFont#1#2#3#4#5{%
  \reset@font\fontsize{#1}{#2pt}%
  \fontfamily{#3}\fontseries{#4}\fontshape{#5}%
  \selectfont}%
\fi\endgroup%
\begin{picture}(9047,10416)(402,-9715)
\put(6901,-2911){\makebox(0,0)[lb]{\smash{{\SetFigFont{9}{10.8}{\rmdefault}{\mddefault}{\updefault}{\color[rgb]{0,0,0}3}%
}}}}
\put(7801,-2911){\makebox(0,0)[lb]{\smash{{\SetFigFont{9}{10.8}{\rmdefault}{\mddefault}{\updefault}{\color[rgb]{0,0,0}3}%
}}}}
\put(7351,-2386){\makebox(0,0)[lb]{\smash{{\SetFigFont{9}{10.8}{\rmdefault}{\mddefault}{\updefault}{\color[rgb]{0,0,0}4}%
}}}}
\put(6751,-3736){\makebox(0,0)[lb]{\smash{{\SetFigFont{9}{10.8}{\rmdefault}{\mddefault}{\updefault}{\color[rgb]{0,0,0}$8\leq k_8\leq 11$}%
}}}}
\put(7276,-3061){\makebox(0,0)[lb]{\smash{{\SetFigFont{9}{10.8}{\rmdefault}{\mddefault}{\updefault}{\color[rgb]{0,0,0}$v$}%
}}}}
\put(8026,-1486){\makebox(0,0)[lb]{\smash{{\SetFigFont{9}{10.8}{\rmdefault}{\mddefault}{\updefault}{\color[rgb]{0,0,0}$k_2\geq 7$}%
}}}}
\put(5776,-2836){\makebox(0,0)[lb]{\smash{{\SetFigFont{9}{10.8}{\rmdefault}{\mddefault}{\updefault}{\color[rgb]{0,0,0}$k_3=4$}%
}}}}
\put(7501,-1561){\makebox(0,0)[lb]{\smash{{\SetFigFont{9}{10.8}{\rmdefault}{\mddefault}{\updefault}{\color[rgb]{0,0,0}$u_5$}%
}}}}
\put(6226,-1486){\makebox(0,0)[lb]{\smash{{\SetFigFont{9}{10.8}{\rmdefault}{\mddefault}{\updefault}{\color[rgb]{0,0,0}$k_1\geq 7$}%
}}}}
\put(6676,-2386){\makebox(0,0)[lb]{\smash{{\SetFigFont{9}{10.8}{\rmdefault}{\mddefault}{\updefault}{\color[rgb]{0,0,0}$u_1$}%
}}}}
\put(7876,-2386){\makebox(0,0)[lb]{\smash{{\SetFigFont{9}{10.8}{\rmdefault}{\mddefault}{\updefault}{\color[rgb]{0,0,0}$u_2$}%
}}}}
\put(5401,-4036){\makebox(0,0)[lb]{\smash{{\SetFigFont{9}{10.8}{\rmdefault}{\mddefault}{\updefault}{\color[rgb]{0,0,0}$k_5=3$}%
}}}}
\put(8626,-4036){\makebox(0,0)[lb]{\smash{{\SetFigFont{9}{10.8}{\rmdefault}{\mddefault}{\updefault}{\color[rgb]{0,0,0}$k_6=3$}%
}}}}
\put(8401,-2836){\makebox(0,0)[lb]{\smash{{\SetFigFont{9}{10.8}{\rmdefault}{\mddefault}{\updefault}{\color[rgb]{0,0,0}$k_4=4$}%
}}}}
\put(7126,-511){\makebox(0,0)[lb]{\smash{{\SetFigFont{9}{10.8}{\rmdefault}{\mddefault}{\updefault}{\color[rgb]{0,0,0}$k_7=4$}%
}}}}
\put(2026,-2836){\makebox(0,0)[lb]{\smash{{\SetFigFont{9}{10.8}{\rmdefault}{\mddefault}{\updefault}{\color[rgb]{0,0,0}3}%
}}}}
\put(2926,-2836){\makebox(0,0)[lb]{\smash{{\SetFigFont{9}{10.8}{\rmdefault}{\mddefault}{\updefault}{\color[rgb]{0,0,0}3}%
}}}}
\put(2476,-2311){\makebox(0,0)[lb]{\smash{{\SetFigFont{9}{10.8}{\rmdefault}{\mddefault}{\updefault}{\color[rgb]{0,0,0}4}%
}}}}
\put(1876,-3661){\makebox(0,0)[lb]{\smash{{\SetFigFont{9}{10.8}{\rmdefault}{\mddefault}{\updefault}{\color[rgb]{0,0,0}$8\leq k_8\leq 11$}%
}}}}
\put(2401,-2986){\makebox(0,0)[lb]{\smash{{\SetFigFont{9}{10.8}{\rmdefault}{\mddefault}{\updefault}{\color[rgb]{0,0,0}$v$}%
}}}}
\put(3151,-1411){\makebox(0,0)[lb]{\smash{{\SetFigFont{9}{10.8}{\rmdefault}{\mddefault}{\updefault}{\color[rgb]{0,0,0}$k_2\geq 7$}%
}}}}
\put(901,-2761){\makebox(0,0)[lb]{\smash{{\SetFigFont{9}{10.8}{\rmdefault}{\mddefault}{\updefault}{\color[rgb]{0,0,0}$k_3=4$}%
}}}}
\put(2626,-1486){\makebox(0,0)[lb]{\smash{{\SetFigFont{9}{10.8}{\rmdefault}{\mddefault}{\updefault}{\color[rgb]{0,0,0}$u_5$}%
}}}}
\put(1351,-1411){\makebox(0,0)[lb]{\smash{{\SetFigFont{9}{10.8}{\rmdefault}{\mddefault}{\updefault}{\color[rgb]{0,0,0}$k_1\geq 7$}%
}}}}
\put(1801,-2311){\makebox(0,0)[lb]{\smash{{\SetFigFont{9}{10.8}{\rmdefault}{\mddefault}{\updefault}{\color[rgb]{0,0,0}$u_1$}%
}}}}
\put(3001,-2311){\makebox(0,0)[lb]{\smash{{\SetFigFont{9}{10.8}{\rmdefault}{\mddefault}{\updefault}{\color[rgb]{0,0,0}$u_2$}%
}}}}
\put(2251,-586){\makebox(0,0)[lb]{\smash{{\SetFigFont{9}{10.8}{\rmdefault}{\mddefault}{\updefault}{\color[rgb]{0,0,0}$k_7=3$}%
}}}}
\put(526,-3961){\makebox(0,0)[lb]{\smash{{\SetFigFont{9}{10.8}{\rmdefault}{\mddefault}{\updefault}{\color[rgb]{0,0,0}$k_5=3$}%
}}}}
\put(3751,-3961){\makebox(0,0)[lb]{\smash{{\SetFigFont{9}{10.8}{\rmdefault}{\mddefault}{\updefault}{\color[rgb]{0,0,0}$k_6=3$}%
}}}}
\put(3526,-2761){\makebox(0,0)[lb]{\smash{{\SetFigFont{9}{10.8}{\rmdefault}{\mddefault}{\updefault}{\color[rgb]{0,0,0}$k_4=4$}%
}}}}
\put(4501,-8161){\makebox(0,0)[lb]{\smash{{\SetFigFont{9}{10.8}{\rmdefault}{\mddefault}{\updefault}{\color[rgb]{0,0,0}3}%
}}}}
\put(5401,-8161){\makebox(0,0)[lb]{\smash{{\SetFigFont{9}{10.8}{\rmdefault}{\mddefault}{\updefault}{\color[rgb]{0,0,0}3}%
}}}}
\put(4951,-7636){\makebox(0,0)[lb]{\smash{{\SetFigFont{9}{10.8}{\rmdefault}{\mddefault}{\updefault}{\color[rgb]{0,0,0}4}%
}}}}
\put(4351,-8986){\makebox(0,0)[lb]{\smash{{\SetFigFont{9}{10.8}{\rmdefault}{\mddefault}{\updefault}{\color[rgb]{0,0,0}$8\leq k_8\leq 11$}%
}}}}
\put(4876,-8311){\makebox(0,0)[lb]{\smash{{\SetFigFont{9}{10.8}{\rmdefault}{\mddefault}{\updefault}{\color[rgb]{0,0,0}$v$}%
}}}}
\put(5626,-6736){\makebox(0,0)[lb]{\smash{{\SetFigFont{9}{10.8}{\rmdefault}{\mddefault}{\updefault}{\color[rgb]{0,0,0}$k_2\geq 7$}%
}}}}
\put(3376,-8086){\makebox(0,0)[lb]{\smash{{\SetFigFont{9}{10.8}{\rmdefault}{\mddefault}{\updefault}{\color[rgb]{0,0,0}$k_3=4$}%
}}}}
\put(5101,-6811){\makebox(0,0)[lb]{\smash{{\SetFigFont{9}{10.8}{\rmdefault}{\mddefault}{\updefault}{\color[rgb]{0,0,0}$u_5$}%
}}}}
\put(3826,-6736){\makebox(0,0)[lb]{\smash{{\SetFigFont{9}{10.8}{\rmdefault}{\mddefault}{\updefault}{\color[rgb]{0,0,0}$k_1\geq 7$}%
}}}}
\put(4276,-7636){\makebox(0,0)[lb]{\smash{{\SetFigFont{9}{10.8}{\rmdefault}{\mddefault}{\updefault}{\color[rgb]{0,0,0}$u_1$}%
}}}}
\put(5476,-7636){\makebox(0,0)[lb]{\smash{{\SetFigFont{9}{10.8}{\rmdefault}{\mddefault}{\updefault}{\color[rgb]{0,0,0}$u_2$}%
}}}}
\put(3001,-9286){\makebox(0,0)[lb]{\smash{{\SetFigFont{9}{10.8}{\rmdefault}{\mddefault}{\updefault}{\color[rgb]{0,0,0}$k_5=3$}%
}}}}
\put(6226,-9286){\makebox(0,0)[lb]{\smash{{\SetFigFont{9}{10.8}{\rmdefault}{\mddefault}{\updefault}{\color[rgb]{0,0,0}$k_6=3$}%
}}}}
\put(6001,-8086){\makebox(0,0)[lb]{\smash{{\SetFigFont{9}{10.8}{\rmdefault}{\mddefault}{\updefault}{\color[rgb]{0,0,0}$k_4=4$}%
}}}}
\put(4651,-5836){\makebox(0,0)[lb]{\smash{{\SetFigFont{9}{10.8}{\rmdefault}{\mddefault}{\updefault}{\color[rgb]{0,0,0}$k_7\geq5$}%
}}}}
\end{picture}%

%% file: fig2new.pstex_t
\begin{picture}(0,0)%
\includegraphics{fig2new.pstex}%
\end{picture}%
\setlength{\unitlength}{2960sp}%
\begingroup\makeatletter\ifx\SetFigFont\undefined%
\gdef\SetFigFont#1#2#3#4#5{%
  \reset@font\fontsize{#1}{#2pt}%
  \fontfamily{#3}\fontseries{#4}\fontshape{#5}%
  \selectfont}%
\fi\endgroup%
\begin{picture}(4074,3624)(1639,-4648)
\put(3676,-1936){\makebox(0,0)[lb]{\smash{{\SetFigFont{11}{13.2}{\rmdefault}{\mddefault}{\updefault}{\color[rgb]{0,0,0}5}%
}}}}
\put(3151,-2686){\makebox(0,0)[lb]{\smash{{\SetFigFont{11}{13.2}{\rmdefault}{\mddefault}{\updefault}{\color[rgb]{0,0,0}3}%
}}}}
\put(4201,-2686){\makebox(0,0)[lb]{\smash{{\SetFigFont{11}{13.2}{\rmdefault}{\mddefault}{\updefault}{\color[rgb]{0,0,0}3}%
}}}}
\put(4951,-2686){\makebox(0,0)[lb]{\smash{{\SetFigFont{11}{13.2}{\rmdefault}{\mddefault}{\updefault}{\color[rgb]{0,0,0}$k_4$}%
}}}}
\put(3601,-2836){\makebox(0,0)[lb]{\smash{{\SetFigFont{11}{13.2}{\rmdefault}{\mddefault}{\updefault}{\color[rgb]{0,0,0}$v$}%
}}}}
\put(4951,-3961){\makebox(0,0)[lb]{\smash{{\SetFigFont{11}{13.2}{\rmdefault}{\mddefault}{\updefault}{\color[rgb]{0,0,0}$k_6$}%
}}}}
\put(2101,-3886){\makebox(0,0)[lb]{\smash{{\SetFigFont{11}{13.2}{\rmdefault}{\mddefault}{\updefault}{\color[rgb]{0,0,0}$k_5$}%
}}}}
\put(3676,-3511){\makebox(0,0)[lb]{\smash{{\SetFigFont{11}{13.2}{\rmdefault}{\mddefault}{\updefault}{\color[rgb]{0,0,0}7}%
}}}}
\put(2626,-1561){\makebox(0,0)[lb]{\smash{{\SetFigFont{11}{13.2}{\rmdefault}{\mddefault}{\updefault}{\color[rgb]{0,0,0}$k_1$}%
}}}}
\put(4576,-1561){\makebox(0,0)[lb]{\smash{{\SetFigFont{11}{13.2}{\rmdefault}{\mddefault}{\updefault}{\color[rgb]{0,0,0}$k_2$}%
}}}}
\put(2176,-2686){\makebox(0,0)[lb]{\smash{{\SetFigFont{11}{13.2}{\rmdefault}{\mddefault}{\updefault}{\color[rgb]{0,0,0}$k_3$}%
}}}}
\put(2671,-2251){\makebox(0,0)[lb]{\smash{{\SetFigFont{11}{13.2}{\rmdefault}{\mddefault}{\updefault}{\color[rgb]{0,0,0}$u_1$}%
}}}}
\put(4546,-2236){\makebox(0,0)[lb]{\smash{{\SetFigFont{11}{13.2}{\rmdefault}{\mddefault}{\updefault}{\color[rgb]{0,0,0}$u_2$}%
}}}}
\put(4756,-3181){\makebox(0,0)[lb]{\smash{{\SetFigFont{11}{13.2}{\rmdefault}{\mddefault}{\updefault}{\color[rgb]{0,0,0}$u_4$}%
}}}}
\put(2446,-3121){\makebox(0,0)[lb]{\smash{{\SetFigFont{11}{13.2}{\rmdefault}{\mddefault}{\updefault}{\color[rgb]{0,0,0}$u_3$}%
}}}}
\end{picture}%

%% file: case23.pstex_t
\begin{picture}(0,0)%
\includegraphics{case23.pstex}%
\end{picture}%
\setlength{\unitlength}{2447sp}%
\begingroup\makeatletter\ifx\SetFigFont\undefined%
\gdef\SetFigFont#1#2#3#4#5{%
  \reset@font\fontsize{#1}{#2pt}%
  \fontfamily{#3}\fontseries{#4}\fontshape{#5}%
  \selectfont}%
\fi\endgroup%
\begin{picture}(9538,8691)(1162,-8665)
\put(3451,-1336){\makebox(0,0)[lb]{\smash{{\SetFigFont{9}{10.8}{\rmdefault}{\mddefault}{\updefault}{\color[rgb]{0,0,0}5}%
}}}}
\put(2926,-2086){\makebox(0,0)[lb]{\smash{{\SetFigFont{9}{10.8}{\rmdefault}{\mddefault}{\updefault}{\color[rgb]{0,0,0}3}%
}}}}
\put(3976,-2086){\makebox(0,0)[lb]{\smash{{\SetFigFont{9}{10.8}{\rmdefault}{\mddefault}{\updefault}{\color[rgb]{0,0,0}3}%
}}}}
\put(1876,-3286){\makebox(0,0)[lb]{\smash{{\SetFigFont{9}{10.8}{\rmdefault}{\mddefault}{\updefault}{\color[rgb]{0,0,0}$k_5$}%
}}}}
\put(3451,-2911){\makebox(0,0)[lb]{\smash{{\SetFigFont{9}{10.8}{\rmdefault}{\mddefault}{\updefault}{\color[rgb]{0,0,0}7}%
}}}}
\put(1726,-2011){\makebox(0,0)[lb]{\smash{{\SetFigFont{9}{10.8}{\rmdefault}{\mddefault}{\updefault}{\color[rgb]{0,0,0}$k_3\geq8$}%
}}}}
\put(2476,-1561){\makebox(0,0)[lb]{\smash{{\SetFigFont{9}{10.8}{\rmdefault}{\mddefault}{\updefault}{\color[rgb]{0,0,0}$u_1$}%
}}}}
\put(2176,-1111){\makebox(0,0)[lb]{\smash{{\SetFigFont{9}{10.8}{\rmdefault}{\mddefault}{\updefault}{\color[rgb]{0,0,0}$k_1=3$}%
}}}}
\put(3451,-2236){\makebox(0,0)[lb]{\smash{{\SetFigFont{9}{10.8}{\rmdefault}{\mddefault}{\updefault}{\color[rgb]{0,0,0}$v$}%
}}}}
\put(2176,-2536){\makebox(0,0)[lb]{\smash{{\SetFigFont{9}{10.8}{\rmdefault}{\mddefault}{\updefault}{\color[rgb]{0,0,0}$u_3$}%
}}}}
\put(3451,-5986){\makebox(0,0)[lb]{\smash{{\SetFigFont{9}{10.8}{\rmdefault}{\mddefault}{\updefault}{\color[rgb]{0,0,0}5}%
}}}}
\put(2926,-6736){\makebox(0,0)[lb]{\smash{{\SetFigFont{9}{10.8}{\rmdefault}{\mddefault}{\updefault}{\color[rgb]{0,0,0}3}%
}}}}
\put(3976,-6736){\makebox(0,0)[lb]{\smash{{\SetFigFont{9}{10.8}{\rmdefault}{\mddefault}{\updefault}{\color[rgb]{0,0,0}3}%
}}}}
\put(3451,-7561){\makebox(0,0)[lb]{\smash{{\SetFigFont{9}{10.8}{\rmdefault}{\mddefault}{\updefault}{\color[rgb]{0,0,0}7}%
}}}}
\put(2476,-6211){\makebox(0,0)[lb]{\smash{{\SetFigFont{9}{10.8}{\rmdefault}{\mddefault}{\updefault}{\color[rgb]{0,0,0}$u_1$}%
}}}}
\put(3451,-6886){\makebox(0,0)[lb]{\smash{{\SetFigFont{9}{10.8}{\rmdefault}{\mddefault}{\updefault}{\color[rgb]{0,0,0}$v$}%
}}}}
\put(2176,-7186){\makebox(0,0)[lb]{\smash{{\SetFigFont{9}{10.8}{\rmdefault}{\mddefault}{\updefault}{\color[rgb]{0,0,0}$u_3$}%
}}}}
\put(1801,-6661){\makebox(0,0)[lb]{\smash{{\SetFigFont{9}{10.8}{\rmdefault}{\mddefault}{\updefault}{\color[rgb]{0,0,0}$k_3=4$}%
}}}}
\put(2026,-5461){\makebox(0,0)[lb]{\smash{{\SetFigFont{9}{10.8}{\rmdefault}{\mddefault}{\updefault}{\color[rgb]{0,0,0}$k_1\geq5$}%
}}}}
\put(1726,-8086){\makebox(0,0)[lb]{\smash{{\SetFigFont{9}{10.8}{\rmdefault}{\mddefault}{\updefault}{\color[rgb]{0,0,0}$k_5$}%
}}}}
\put(8851,-1636){\makebox(0,0)[lb]{\smash{{\SetFigFont{9}{10.8}{\rmdefault}{\mddefault}{\updefault}{\color[rgb]{0,0,0}5}%
}}}}
\put(8326,-2386){\makebox(0,0)[lb]{\smash{{\SetFigFont{9}{10.8}{\rmdefault}{\mddefault}{\updefault}{\color[rgb]{0,0,0}3}%
}}}}
\put(9376,-2386){\makebox(0,0)[lb]{\smash{{\SetFigFont{9}{10.8}{\rmdefault}{\mddefault}{\updefault}{\color[rgb]{0,0,0}3}%
}}}}
\put(8851,-3211){\makebox(0,0)[lb]{\smash{{\SetFigFont{9}{10.8}{\rmdefault}{\mddefault}{\updefault}{\color[rgb]{0,0,0}7}%
}}}}
\put(7876,-1861){\makebox(0,0)[lb]{\smash{{\SetFigFont{9}{10.8}{\rmdefault}{\mddefault}{\updefault}{\color[rgb]{0,0,0}$u_1$}%
}}}}
\put(8851,-2536){\makebox(0,0)[lb]{\smash{{\SetFigFont{9}{10.8}{\rmdefault}{\mddefault}{\updefault}{\color[rgb]{0,0,0}$v$}%
}}}}
\put(7576,-2836){\makebox(0,0)[lb]{\smash{{\SetFigFont{9}{10.8}{\rmdefault}{\mddefault}{\updefault}{\color[rgb]{0,0,0}$u_3$}%
}}}}
\put(7201,-2311){\makebox(0,0)[lb]{\smash{{\SetFigFont{9}{10.8}{\rmdefault}{\mddefault}{\updefault}{\color[rgb]{0,0,0}$k_3\geq6$}%
}}}}
\put(7576,-1111){\makebox(0,0)[lb]{\smash{{\SetFigFont{9}{10.8}{\rmdefault}{\mddefault}{\updefault}{\color[rgb]{0,0,0}$k_1=4$}%
}}}}
\put(7201,-3586){\makebox(0,0)[lb]{\smash{{\SetFigFont{9}{10.8}{\rmdefault}{\mddefault}{\updefault}{\color[rgb]{0,0,0}$k_5$}%
}}}}
\put(9001,-6061){\makebox(0,0)[lb]{\smash{{\SetFigFont{9}{10.8}{\rmdefault}{\mddefault}{\updefault}{\color[rgb]{0,0,0}5}%
}}}}
\put(8476,-6811){\makebox(0,0)[lb]{\smash{{\SetFigFont{9}{10.8}{\rmdefault}{\mddefault}{\updefault}{\color[rgb]{0,0,0}3}%
}}}}
\put(9526,-6811){\makebox(0,0)[lb]{\smash{{\SetFigFont{9}{10.8}{\rmdefault}{\mddefault}{\updefault}{\color[rgb]{0,0,0}3}%
}}}}
\put(9001,-7636){\makebox(0,0)[lb]{\smash{{\SetFigFont{9}{10.8}{\rmdefault}{\mddefault}{\updefault}{\color[rgb]{0,0,0}7}%
}}}}
\put(8026,-6286){\makebox(0,0)[lb]{\smash{{\SetFigFont{9}{10.8}{\rmdefault}{\mddefault}{\updefault}{\color[rgb]{0,0,0}$u_1$}%
}}}}
\put(9001,-6961){\makebox(0,0)[lb]{\smash{{\SetFigFont{9}{10.8}{\rmdefault}{\mddefault}{\updefault}{\color[rgb]{0,0,0}$v$}%
}}}}
\put(7726,-7261){\makebox(0,0)[lb]{\smash{{\SetFigFont{9}{10.8}{\rmdefault}{\mddefault}{\updefault}{\color[rgb]{0,0,0}$u_3$}%
}}}}
\put(7576,-5536){\makebox(0,0)[lb]{\smash{{\SetFigFont{9}{10.8}{\rmdefault}{\mddefault}{\updefault}{\color[rgb]{0,0,0}$k_1\geq5$}%
}}}}
\put(7276,-8161){\makebox(0,0)[lb]{\smash{{\SetFigFont{9}{10.8}{\rmdefault}{\mddefault}{\updefault}{\color[rgb]{0,0,0}$k_5$}%
}}}}
\put(7276,-6736){\makebox(0,0)[lb]{\smash{{\SetFigFont{9}{10.8}{\rmdefault}{\mddefault}{\updefault}{\color[rgb]{0,0,0}$k_3\geq5$}%
}}}}
\end{picture}%

%% file: fig4.pstex_t
\begin{picture}(0,0)%
\includegraphics{fig4.pstex}%
\end{picture}%
\setlength{\unitlength}{2763sp}%
\begingroup\makeatletter\ifx\SetFigFont\undefined%
\gdef\SetFigFont#1#2#3#4#5{%
  \reset@font\fontsize{#1}{#2pt}%
  \fontfamily{#3}\fontseries{#4}\fontshape{#5}%
  \selectfont}%
\fi\endgroup%
\begin{picture}(11274,9880)(-11,-9399)
\end{picture}%